\documentclass{amsart}

\usepackage{amssymb}

\newtheorem{thm}{Theorem}[section]

\newtheorem{lem}[thm]{Lemma}

\theoremstyle{definition}
\newtheorem{definition}[thm]{Definition}

\theoremstyle{remark}
\newtheorem{remark}[thm]{Remark}
\numberwithin{equation}{section}

\begin{document}

\title{Hinged beam dynamics}

\author{David Raske}
\address{1210 Washtenaw, Ypsilanti, MI, 48197}
\email{nonlinear.problem.solver@gmail.com}
\subjclass[2010]{Primary 35L35; Secondary 35Q99, 35L76}
\keywords{nonlinear partial differential equations; Galerkin method; continuum mechanics}

\begin{abstract}
In this paper we prove large-time existence and uniqueness of high regularity weak solutions to some initial/boundary value problems involving a nonlinear fourth order wave equation. These sorts of problems arise naturally in the study of vibrations in beams that are hinged at both ends. The method used to prove large-time existence is the Galerkin approximation method.

\end{abstract}                        

\maketitle

\section{Introduction} 

Let $a$ and $b$ be two real numbers such that $a < b$. Let $\Omega$ be the open interval $(a,b)$. Let $T$ be a positive real number. Let $u = u(x,t)$ be a real-valued function defined on $\Omega \times (0,T)$. Let $u_0 = u_0(x)$ and $u_1 = u_1(x)$ be two real-valued functions that are defined on $\Omega$. Let $F_1$ and $F_2$ be two real-valued functions defined on $\mathbb{R}$. Let $f=f(x,t)$ be a real-valued function defined on $\Omega \times (0,T)$. The initial/boundary value problem
\begin{equation}\label{IBVP}
\begin{split}
&  u_{tt} + F_1(u_t) + u_{xxxx} + F_2(u) = f \text{ on } \Omega \times (0,T), \\ & u(a,t) = u(b,t) = 0 \text{ for all } t \in (0,T), \\ & u_{xx}(a,t) = u_{xx}(b,t) = 0 \text{ for all } t \in (0,T), \\ & u(x,0) = u_0(x) \text{ for all } x \in \Omega, \\ & u_t(x,0) = u_1(x) \text{ for all } x \in \Omega.
\end{split}
\end{equation}
occurs naturally in the study of vibrations in beams that are hinged at both ends. The goal of this paper is to find conditions on $F_1$, $F_2$, and the data of the above problem that guarantee the existence and uniqueness of solutions.

Before we state the main result of the paper, we need to define some function spaces.  Let  $H^2_*(\Omega):= \{ u \in H^2(\Omega) | u \in H^1_0(\Omega) \}$.  Let $H^4_*(\Omega):=\{u \in H^4(\Omega)| u \in H^1_0(\Omega) \text{ and such that } u_{xx} \in H^1_0(\Omega)\}$. We will endow $H^2_*(\Omega)$ with the inner-product 
$$
(u,v)_{H^2_*} = \int_{\Omega} u_{xx}v_{xx} \, dx,
$$
and endow $H^4_*(\Omega)$ with the inner-product
$$
(u,v)_{H^4_*} = \int_{\Omega} u_{xxxx}v_{xxxx} \, dx.
$$
We will see later that these inner-products make $H^2_*(\Omega)$ and $H^4_*(\Omega)$ Hilbert spaces. We will also say that a function $u$ is a {\it high regularity weak solution} of the initial/boundary value problem \eqref{IBVP} if (a) $u$ belongs to 
$$
L^\infty(0,T;H_*^4(\Omega)), 
$$
$$
W^{1,\infty}(0,T;H_*^2(\Omega)),
$$
and 
$$
W^{2,\infty}(0,T;L^2(\Omega));
$$
(b) $u$ satisfies the initial conditions and boundary value conditions of \eqref{IBVP}; and (c) $u$ satisfies the equation
\begin{equation}\label{strong}
(u''(t),v)_{L^2} + (u(t), v)_{H^2_*} + (F_1(u'(t)), v)_{L^2} + (F_2(u(t)),v)_{L^2} = (f,v)_{L^2}
\end{equation}
for all $v \in H^2_*(\Omega)$ and for a.e. $t \in [0,T]$. (Here $(\cdot,\cdot)_{L^2}$ is the standard inner-product on $L^2(\Omega)$.) Note, as well, that the initial conditions contained in \eqref{IBVP} make sense since $u \in C([0,T]; H^2_*(\Omega))$ and $u_t \in C([0,T]; L^2(\Omega))$.

We can now state the main result of the paper.

\begin{thm} Let $a$ and $b$ be two real numbers, with $a < b$. Let $T$ be a positive real number. Let $\Omega$ be the open interval $(a,b)$. Let $f$ be a $C^1$ map from $[0,T]$ into $L^2(\Omega)$. Let $F_1$ be a nondecreasing $C^1$ function from $\mathbb{R}$ into $\mathbb{R}$, with $F_1(0)=0$.  Let $F_2$ be a $C^1$ function from $\mathbb{R}$ into $\mathbb{R}$ such that there exists a non-positive real number $c$ such that $\int_{0}^s F_2(t) dt \geq c$ for all $s \in \mathbb{R}$. Furthermore, let $u_0$ be an element of $H_*^4(\Omega)$ and let $u_1$ be an element of $H^2_*(\Omega)$. Then there exists a unique high regularity weak solution of the initial/boundary value problem \eqref{IBVP}.
\end{thm}

The method used to prove existence of high regularity weak solutions is the Galerkin method. The main source of difficulty is the possibility that $u'$ shows up inside a nonlinear function. We overcome this problem by obtaining stronger bounds on the Galerkin approximations than usual. This is accomplished by differentiating the equations that the finite dimensional approximations satisfy and using the fact that $F_1$ is an increasing function. Uniqueness follows from an application of the mean value theorem and Gronwall's inequality to the problem.

The organization of this paper is as follows. Section two consists of remarks and lemmas that will be used in the proof of Theorem 1.1. Section three consists of the proof of Theorem 1.1. The proof has been divided into seven steps to make it easier to read.

\section{Preliminary Material}
\subsection{Higher order Hilbert spaces}

Let $a$ and $b$ be two real numbers such that $a < b$. Let $\Omega$ be the open interval $(a,b)$. Let $\{ \lambda_i \}_{i=1}^\infty$ be the set of eigenvalues for the eigenvalue problem
\begin{equation}\label{eigenfunctions}
\begin{split}
-u_{xx} & = \lambda u \quad \text{on} \quad \Omega \\
       u(a) & =0 \\
        u(b) &=0,   
\end{split}
\end{equation}
ordered in the usual manner. Note that all of the eigenvalues are positive. Let $\{e_i\}_{i=1}^\infty$ be the corresponding set of eigenfunctions. We will assume that they are normalized with respect to the $L^2(\Omega)$ norm, $||\cdot||_{L^2}$. Due to the spectral theory of symmetric, compact operators, we can assume that $e_i$ is orthogonal with respect to the $L^2(\Omega)$ inner-product to $e_j$ if $i \neq j$.
Note as well that all of the eigenfunctions belong to $H_*^2(\Omega)$ and $H_*^4(\Omega)$. As we'll now see we can say a lot more about $\{e_i\}_{i=1}^\infty$, $H_*^2(\Omega)$, and $H_*^4(\Omega)$.

First, we need the following

\begin{definition} 
Let $u$ be a summable function on the open interval $\Omega$. Then we will let $(u)_\Omega$ denote the average of $u$ on $\Omega$.
\end{definition}

\begin{lem}
Let $a$ and $b$ be two real numbers with $a < b$. Let $\Omega$ be the open interval $(a,b)$. Then (i) there exists a positive real number $C$ such that for all $u \in H^2_*(\Omega)$ we have that $||u||_{H^2_*}^2 \geq C ||u||_{L^2}^2$. We also have that (ii) there exists a positive real number $C$ such that for all $u \in H^4_*(\Omega)$, $||u||_{H^4_*}^2 \geq C ||u||_{L^2}^2$.
\end{lem}

\begin{proof}
First let us prove assertion (i). Invoking the Poincar\'e inequality, we have the existence of a positive real number $C$ such that
\begin{equation}\label{Poincare}
\int_\Omega (u_{xx})^2 \, dx \geq C \int_\Omega (u_{x} - (u_{x})_\Omega)^2 \, dx
\end{equation}
for all $u \in H^2_*(\Omega)$. Since $u$ vanishes on the boundary of $\Omega$ the above inequality is equivalent to the existence of a positive real number $C$ such that
\begin{equation}\label{Poincare 2}
\int_\Omega (u_{xx})^2 \, dx \geq C \int_\Omega (u_{x})^2 \, dx
\end{equation}
for all $u \in H^2_*(\Omega)$. Invoking the Poincar\'e inequality again, we have the existence of a positive real number $C$ such that
\begin{equation}\label{Poincare 3}
\int_\Omega (u_{x})^2 \, dx \geq C \int_\Omega u^2 \, dx
\end{equation}
for all $u \in H^1_0(\Omega)$. Combining \eqref{Poincare 2} with \eqref{Poincare 3}, we obtain assertion (i).

Now let us prove assertion (ii). Invoking the Poincar\'e inequality we have the existence of a positive real number $C$ such that
\begin{equation}\label{Poincare 4}
\int_\Omega (u_{xxxx})^2 \, dx \geq C \int_\Omega (u_{xxx} - (u_{xxx})_\Omega)^2 \, dx
\end{equation}
for all $u \in H^4_*(\Omega)$. Since $u_{xx}$ vanishes on the boundary of $\Omega$ the above inequality is equivalent to the existence of a positive real number $C$ such that
\begin{equation}\label{Poincare 5}
\int_\Omega (u_{xxxx})^2 \, dx \geq C \int_\Omega (u_{xxx})^2 \, dx
\end{equation}
for all $u \in H^4_*(\Omega)$. Invoking the Poincar\'e inequality again, we have the existence of a positive real number $C$ such that
\begin{equation}\label{Poincare 6}
\int_\Omega (u_{xxx})^2 \, dx \geq C \int_\Omega (u_{xx})^2 \, dx
\end{equation}
for all $u \in H^4_*(\Omega)$. Combining \eqref{Poincare 5} with \eqref{Poincare 6} we have the existence of a positive real number $C$ such that
\begin{equation}\label{Poincare 7}
\int_\Omega (u_{xxxx})^2 \, dx \geq C \int_\Omega (u_{xx})^2 \, dx
\end{equation}
for all $u \in H^4_*(\Omega)$. Recalling assertion (i), we have assertion (ii).

\end{proof}

\begin{remark} Due to the intermediate derivatives theorem (see Theorem 4.15 of \cite{A}), Lemma 2.2 implies that the $H^2_*(\Omega)$ norm is equivalent to the usual $H^2(\Omega)$ norm provided we restrict the latter norm to the linear subspace $H^2_*(\Omega)$. Lemma 2.2 also implies that the $H^4_*(\Omega)$ norm is equivalent to the usual $H^4(\Omega)$ norm provided that we restrict the latter norm to the linear subspace $H^4_*(\Omega)$. In both cases we have completeness of the respective inner product spaces.
\end{remark}

\begin{lem} (i) $\{e_i\}_{i=1}^\infty$ forms an orthogonal basis for $H_*^2(\Omega)$, and (ii) $\{e_i\}_{i=1}^\infty$ forms a orthogonal basis for $H_*^4(\Omega)$.
\end{lem}

\begin{proof} First let us show that $e_i$ is orthogonal to $e_j$ if $i \neq j$ with respect to the $H^2_*(\Omega)$ inner product. Recalling the definition of the sequence $\{e_i\}_{i=1}^\infty$, we have that $-(e_i)_{xx} = \lambda_i e_i$ for all $i \in \mathbb{N}$. It follows that
\begin{equation*}
\begin{split}
(e_i,e_j)_{H^2_*} & = ((e_i)_{xx}, (e_j)_{xx})_{L^2} \\
                                & = \lambda_i \lambda_j (e_i,e_j)_{L^2} \\
                               & = \lambda_i \lambda_j \delta_{ij},
\end{split}
\end{equation*}
where $\delta_{ij}$ is the Kronecker delta. Thus we have the desired orthogonality property for $H^2_*(\Omega)$.

We will now show that the sequence, $\{e_i\}_{i=1}^\infty$ forms a complete basis for $H^2_*(\Omega)$. Towards this end it suffices to show that if there exists an element of $H^2_*(\Omega)$ such that $(u,e_i)_{H^2_*} = 0$ for all $i \in \mathbb{N}$, then $u \equiv 0$. So let us suppose that there exists a function $u \in H^2_*(\Omega)$ such that $(u,e_i)_{H^2_*} = 0$ for all $i \in \mathbb{N}$. Note that this is equivalent to $(u_{xx}, (e_i)_{xx})_{L^2} = 0$ for all $i \in \mathbb{N}$. Since $u \in H^2_*(\Omega)$ and $e_i$ is in $H^4_*(\Omega)$ for all $i \in \mathbb{N}$ we can integrate by parts. This gives us $(u, (e_i)_{xxxx})_{L^2} = 0$ for all $i \in \mathbb{N}$. Recalling the definition of the $e_i$ we have $(u, e_i)_{L^2}=0$ for all $i \in \mathbb{N}$. Since $\{e_i\}_{i=1}^\infty$ form a complete basis for $L^2(\Omega)$ we have that $u \equiv 0$. 

Now let us turn our attention to the sequence $\{e_i\}_{i=1}^\infty$ with relation to the $H^4_*(\Omega)$ inner-product. We will now see that this sequence is orthogonal with respect to this inner product. First note that $(e_i)_{xxxx} = \lambda_i^2 e_i$ for all $i \in \mathbb{N}$. It follows that $((e_i)_{xxxx}, (e_j)_{xxxx})_{L^2} = \lambda_i^2 \lambda_j^2 (e_i,e_j)_{L^2}$, for all $(i,j) \in \mathbb{N} \times \mathbb{N}$. The orthogonality of this sequence with respect to the $L^2(\Omega)$ norm then gives us our desired result.

It remains then to show that this sequence forms a complete basis for $H^4_*(\Omega)$. Again, it suffices to show that if there exists an element $u \in H^4_*(\Omega)$ such that $(u,e_i)_{H^4_*}=0$ for all $i \in \mathbb{N}$ then $u \equiv 0$. So let us suppose that there exists a function $u \in H^4_*(\Omega)$ such that $(u,e_i)_{H^4_*}=0$ for all $i \in \mathbb{N}$. Note that this is equivalent to $(u_{xxxx}, (e_i)_{xxxx})_{L^2}=0$ for all $i \in \mathbb{N}$. Again, recalling the definition of $\{e_i\}_{i=1}^\infty$ we have $(u_{xxxx}, e_i)_{L^2}=0$ for all $i \in \mathbb{N}$. Since $u \in H^4_*(\Omega)$ and all of the $e_i$ are in $H^2_*(\Omega)$ we can integrate by parts. This gives us that $(u_{xx}, (e_i)_{xx})_{L^2} = 0$ for all $i \in \mathbb{N}$. As is proven above, this is equivalent to $(u, e_i)_{L^2(\Omega)} = 0$ for all $i \in \mathbb{N}$. It follows that $u \equiv 0$. Completeness follows. 

\end{proof}

\subsection{Spaces involving time}

\begin{lem}
Let $X$ be a connected, compact topological space. Let $\{u_k\}_{k=1}^\infty$ be a sequence of continuous functions from $X$ into $\mathbb{R}$. Let $F$ be a continuous function from $\mathbb{R}$ into $\mathbb{R}$. Suppose that there exists a positive real number $C$ such that $||u_k||_{C(X)} \leq C$ for all $k \geq 1$. Then there exists a positive real number $D$ such that $||F(u_k)||_{C(X)} \leq D$ for all $k \geq 1$. 
\end{lem}

\begin{proof}
Let $C$ be a positive real number such that $||u_k||_{C(X)} \leq C$ for all $k \geq 1$. Fix $k \in \mathbb{Z}$. Then there exists a positive real number $D$ that can be chosen independent of $k$ such that 
\begin{equation}
\begin{split}
||F(u_k)||_{C(X)} &= \max_{z \in [\min \; \text{of} \; u_k \; \text{on} \; X, \max \; \text{of} \;  u_k \; \text{on} \;  X]} |F(z)|\\ 
& \leq \max_ {z \in [-||u_k||_{C(X)}, ||u_k||_{C(X)}]} |F(z)| \\
& \leq \max_{z \in [-C, C]} |F(z)| \\
&  \leq D,
\end{split}
\end{equation}
for all $k \geq 1$. The lemma follows.
\end{proof}

\begin{remark} Let $T$ be a positive real number, and let $X$ be a  compact topological space. Let $\{u_k\}_{k=1}^\infty$ be a sequence of elements of $C([0,T];C(X))$. Then the elements of this sequence are also elements of $C(X \times [0,T])$. Furthermore, the property that $\{u_k\}_{i=1}^\infty$ is bounded with respect to the $C([0,T];C(X))$ norm  is equivalent to the property that $\{u_k\}_{k=1}^\infty$ is bounded with respect to the $C(X \times [0,T])$ norm.
\end{remark}

Another result that will be useful in the next section of this paper is the following

\begin{lem} Let $F$ be a $C^1$ function from $\mathbb{R}$ into $\mathbb{R}$. Let $T$ be a positive real number, and let $n$ be a positive integer. Let $U$ be a bounded, open subset of $\mathbb{R}^n$ with a $C^1$ boundary. Let $X$ be a real Banach space that is compactly embedded into $C(\overline{U})$. Let $\{u_k(t)\}_{k=1}^\infty$ be a sequence of elements of $C([0,T]; X)$. Furthermore, suppose that there exists a positive real number $C$ such that $||u_k(t)||_{C([0,T]; X)} \leq C$ for all positive integers $k$. Finally, suppose there exists a $C([0,T]; X)$ function $u$ such that $||u_k - u||_{C([0,T]; X)} \rightarrow 0$ as $k \rightarrow \infty$. Then we have $||F(u_k) - F(u))||_{C([0,T]; L^2(U))} \rightarrow 0$ as $k \rightarrow \infty$.
\end{lem}

\begin{proof} Recalling Remark 2.6, we  put 
\begin{equation}\label{bounded}
M:=\sup_{k \in \mathbb{N}} ||u_k||_{C(\overline{U} \times [0,T])} < \infty.
\end{equation}
On the other hand, we know that $F$ is a $C^1$ function, and in particular it is locally Lipschitz continuous. Hence there exists a positive real number $L_F$ such that
\begin{equation}\label{Lipschitz}
|F(s_1)-F(s_{2})| \leq L_F |s_1 - s_{2}| \quad \text{for any} \quad  s_1, s_{2} \in [-M,M].
\end{equation}
We can now combine \eqref{bounded} wih \eqref{Lipschitz} to see that for all $(k,l) \in \mathbb{N} \times \mathbb{N}$,
\begin{equation}\label{Cauchy}
||F(u_k(t)) - F(u_l(t))||_{L^2} \leq  L_F ||u_k(t) - u_l(t)||_{L^2}, 
\end{equation}
for all $t \in [0,T]$, and hence
\begin{equation}
\max_{t \in [0,T]} ||F(u_k(t)) - F(u_l(t))||_{L^2} \leq L_F \max_{t \in [0,T]} ||u_k(t)-u_l(t)||_{L^2}
\end{equation}
for all $(k,l) \in \mathbb{N} \times \mathbb{N}$. Since the right hand side of the above inequality vanishes as $k$ and $l$ go to infinity we have that the sequence $\{F(u_k(t))\}_{k=1}^\infty$ is a Cauchy sequence with respect to the $C([0,T]; L^2(U))$ norm. It follows that there exists an element of $C([0,T]; L^2(U))$, $Y$, such that $||F(u_k)- Y||_{C([0,T]; L^2(U))}$ vanishes as $k$ goes to infinity.

Now we'll show that $u(t) \in C(\overline{U})$ for all $t \in [0,T]$. First, let $t_* $ be a real number in $[0,T]$. Due to our hypotheses, we know that $||u_k(t_*) - u(t_*)||_{L^2(U)} \rightarrow 0$ as $k \rightarrow \infty$. Due to the compactness of the embedding of $X$ into $C(\overline{U})$, we also have that there exists a subsequence $\{u_{k_l}(t_*)\}_{l=1}^\infty$ of $\{u_k(t_*)\}_{k=1}^\infty$ and a $C(\overline{U})$ function $Z(t_*)$ such that $||u_{k_l}(t_*) - Z(t_*)||_{C(\overline{U})} \rightarrow 0$ as $l \rightarrow \infty$. Since $||u_k(t_*) -  u(t_*)||_{C(\overline{U})} \rightarrow 0$ as $k$ goes to $\infty$, we have that $||u_{k_l}(t_*) - u(t_*)||_{C(\overline{U})} \rightarrow 0$ as $l \rightarrow \infty$. It follows that $Z(t_*) = u(t_*)$, and hence $u(t_*) \in C(\overline{U})$. Since $t_*$ could be any real number in $[0,T]$, we have that $u(t) \in C(\overline{U})$ for all $t \in [0,T]$.  

Now that we know that $u(t) \in C(\overline{U})$ for all $t \in [0,T]$, we can write
\begin{equation}\label{Cauchyb}
||F(u_k(t) - F(u(t))||_{L^2} \leq C ||u_k(t) - u(t)||_{L^2}, 
\end{equation}
where $C$ is a positive real number that does not depend on $t$. It follows that for all $t \in [0,T]$, $||F(u_k(t) - F(u(t))||_{L^2(U)} \rightarrow 0$ as $k \rightarrow \infty$. This in turn allows us to conclude that $Y(t) = u(t)$ for all $t \in [0,T]$. Since $||u_k -  Y||_{C([0,T]; L^2(U))} \rightarrow 0$ as $k \rightarrow \infty$, we have the lemma

\end{proof}

\section{Proof of Theorem 1.1}

\begin{proof}

{\it Step one.} Following \cite{GF}, we proceed as follows. For any $k \geq 1$ let us write $W_k = \text{span} \{e_1, \dots, e_k\}$, where $\{e_i\}_{i=1}^\infty$ is the set of eigenfunctions defined in section two of this paper. Let $\{\lambda_i\}_{i=1}^\infty$ be the set of eigenvalues defined in section two. For any $k \geq 1$ let
\begin{equation}\label{eq 1}
\begin{split}
u_0^k  & := \Sigma_{i=1}^k (u_0,e_i)_{L^2} \, e_i \\
            & = \Sigma_{i=1}^k \frac{(u_0,e_i)_{H_*^2}}{\lambda_i^2} \, e_i \\
            & = \Sigma_{i=1}^k \frac{(u_0,e_i)_{H_*^4}}{\lambda_i^4} \, e_i
\end{split}
\end{equation}
Let us also write
\begin{equation}\label{eq 2}
\begin{split}
u_1^k  & := \Sigma_{i=1}^k (u_1,e_i)_{L^2} \, e_i \\
            & = \Sigma_{i=1}^k \frac{(u_1,e_i)_{H_*^2}}{\lambda_i^2} \, e_i.
\end{split}
\end{equation}
It follows that $u_0^k \rightarrow u_0$ in $H^4_*(\Omega)$ and $u_1^k \rightarrow u_1$ in $H^2_*(\Omega)$ as $k \rightarrow \infty$.

The goal of this step is to establish that for any $k \geq 1$ there exists a unique solution $u_k \in C^3([0,T]; W_k)$ to the variational problem
\begin{equation}\label{eq 3}
\begin{split}
& ({u}''(t),v)_{L^2} + (u(t),v)_{H_*^2} + (F_1({u}'(t)),v)_{L^2} + (F_2(u(t)),v)_{L^2} =  (f(t),v)_{L^2} \\ & u(0) = u_0^k, u'(0) = u_1^k.
\end{split}
\end{equation}
for any $v \in W_k$ and $t \in (0,T)$.

Towards this end, let us first define $H(\Omega)$ to be the dual space of $H^2_*(\Omega)$ and let $<\cdot,\cdot>$ be the corresponding duality. Now, fix $k \in \mathbb{N}$  and make the ansatz
\begin{equation*}
u_k(t) = \Sigma_{i=1}^k g_i^k(t) e_i.
\end{equation*}
Let us also write 
\begin{equation*}
g^k(t) := (g_1^k(t),...,g_k^k(t))^T.
\end{equation*}
We want $u_k$ to solve \eqref{eq 3}. It follows that vector-valued function $g^k$ must solve
\begin{equation}\label{eq 4}
\begin{split}
& ({g^k}(t))'' + \Lambda_k g^k(t)  + \Gamma^1_k((g^k(t))') + \Gamma^2_k(g^k(t)) = G_k(t) \\  
& g^k(0) = u_0^k, ({g^k)}'(0) = u^k_1.
\end{split}
\end{equation}
for all $t \in (0,T)$. Here 
\begin{equation*}
\Lambda_k := \text{diag}(\lambda^2_1, \dots, \lambda^2_k),
\end{equation*}
$\Gamma^i_k$ is a map from $\mathbb{R}^k$ into $\mathbb{R}^k$ for all $i \in \{1,2\}$ defined by
\begin{equation*}
\Gamma^i_k(y_1,\dots,y_k) := ((F_i(\Sigma^k_{j=1} y_j  e_j),e_1)_{L^2}, \dots, (F_i(\Sigma^k_{j=1} y_j e_j),e_k)_{L^2})^T,
\end{equation*}
and
\begin{equation*}
G_k(t) := ((f(t), e_1)_{L^2}, \dots, (f(t), e_k)_{L^2})^T
\end{equation*} 
is a $C^1$ map from $[0,T]$ into $\mathbb{R}^k$. 

Since $F_i$ is a $C^1$ map from $\mathbb{R}$ into $\mathbb{R}$ for all $i \in \{1,2\}$, it follows that $\Gamma_k^i$ is a $C^1$ map from $\mathbb{R}^k$ into $\mathbb{R}^k$ for all $i \in \{1,2\}$. This implies that \eqref{eq 4} admits a unique local solution. This, in turn, allows us to conclude that $u_k(t)$ is a local solution in some maximal interval of continuation $[0,t_k)$, $t_k \in (0,T]$, of the problem
\begin{equation}\label{eq 5}
\begin{split} 
& {u_k''}(t) + L(u_k(t)) + P_k F_1(u_k'(t)) + P_k F_2(u_k(t)) = P_k f(t) \; \text{for any} \; t \in [0,t_k) \\
& u_k(0) = u_0^k, u_k'(0) = u_1^k. 
\end{split}
\end{equation}
where  $P_k: L^2(\Omega) \rightarrow W_k$ is the orthogonal projection onto $W_k$, and where $L: H^2_*(\Omega) \rightarrow H(\Omega)$ is defined by the rule $<Lu,v> = (u,v)_{H^2_*}$ for any $u$ and $v$ in $H^2_*(\Omega)$. It follows that we can write
\begin{equation}\label{eq 6.5}
\begin{split}
& (u_k''(t),e_i)_{L^2} + (u_k(t),e_i)_{H_*^2} + (F_1(u_k'(t)),e_i)_{L^2} + (F_2(u_k(t)),e_i)_{L^2}\\ & =  (f(t),e_i)_{L^2}, \\ & u_k(0) = u_0^k, u_k'(0) = u_1^k.
\end{split}
\end{equation}
for all $t \in (0,t_k)$ and for any $i \in \{1,2,3,...,k\}$. An immediate consequence of the above is that
\begin{equation}\label{eq 6}
\begin{split}
& (u_k''(t),v)_{L^2} + (u_k(t),v)_{H_*^2} + (F_1(u_k'(t)),v)_{L^2} + (F_2(u_k(t)),v)_{L^2}\\& =  (f(t),v)_{L^2}, \\ & u_k(0) = u_0^k, u'_k(0) = u_1^k.
\end{split}
\end{equation}
for all $v \in W_k$ and $t \in (0,t_k)$.

{\it Step two.} The goal of this step is to obtain a uniform bound on the sequence $\{u_k\}_{k=1}^\infty$ that is strong enough to guarantee that the ordinary differential equations defined in step one exist on all of $[0,T]$. Since $u_k'(t) \in W_k$ for all $t \in [0,t_k)$ we can use \eqref{eq 6.5} to show that
\begin{equation}\label{6.75}
\begin{split}
& (u_k''(t), u_k'(t))_{L^2} + (u_k(t),u_k'(t))_{H_*^2} + (F_1(u_k'(t)), u_k'(t))_{L^2} +  \\ & (F_2(u_k(t)),u_k'(t))_{L^2} = (f(t), u_k'(t))_{L^2} 
\end{split}
\end{equation} 
for any $t \in [0,t_k)$. Now, let $f_2 : \mathbb{R} \rightarrow \mathbb{R}$ be defined as follows: $f_2(s) := \int_{0}^s F_2(t) dt$. Furthermore, let  $V: H^2_*(\Omega) \rightarrow \mathbb{R}$ be defined by the rule $u \rightarrow \int_\Omega f_2(u) dx$. We can now write
\begin{equation}\label{eq 7}
\begin{split}
&\frac{1}{2} \frac{d}{dt} (u_k'(t),u_k'(t))_{L^2} + \frac{1}{2} \frac{d}{dt} (u_k(t),u_k(t))_{H_*^2} + \frac{d}{dt} V(u_k(t)) \\ & + (F_1(u_k'(t)),u_k'(t))_{L^2} = (f(t),u_k'(t))_{L^2}. 
\end{split}
\end{equation}
for any $t \in [0,t_k)$. Now, invoking the hypotheses of Theorem 1.1, we see that $F_1(a)a \geq 0$ for all $a \in \mathbb{R}$ and $V(u) \geq c(b-a)$ for all $u \in H^2_*(\Omega)$. Note as well that $(v,v)_{H_*^2} \geq 0$ for all $v \in H_*^2(\Omega)$. Given these observations, we can now conclude that for all $t \in [0,t_k)$
\begin{equation}\label{eq 8}
\begin{split}
&\frac{1}{2} \frac{d}{dt} (u_k'(t),u_k'(t))_{L^2} + \frac{1}{2} \frac{d}{dt} (u_k(t),u_k(t))_{H_*^2} + \frac{d}{dt} (V(u_k(t))-c(b-a)) \\
& \leq \frac{1}{2} (f(t),f(t))_{L^2} + (\frac{1}{2} (u_k'(t),u_k'(t))_{L^2} + \frac{1}{2}(u_k(t),u_k(t))_{H_*^2}\\ & + (V(u_k(t))-c(b-a)) 
\end{split}
\end{equation}
We can now invoke Gronwall's inequality and the non-negativity of $V(\cdot) - c(b-a)$ to conclude that for all $t \in [0,t_k)$
\begin{equation}\label{eq 9}
\begin{split}
&\frac{1}{2}(u_k'(t),u_k'(t))_{L^2} + \frac{1}{2}(u_k(t),u_k(t))_{H_*^2} \\  & \leq e^t(\frac{1}{2} (u_k'(0),u_k'(0))_{L^2} + \frac{1}{2}(u_k(0),u_k(0))_{H_*^2} + (V(u_k(0)) - c(b-a))  \\ & + \frac{1}{2} \int_0^t (f(s),f(s))_{L^2} \, ds) \\ & \leq e^T(\frac{1}{2} (u_k'(0),u_k'(0))_{L^2} + \frac{1}{2}(u_k(0),u_k(0))_{H_*^2} + (V(u_k(0))-c(b-a))  \\ &+  \frac{1}{2} \int_0^T (f(s),f(s))_{L^2} \, ds) \\
&  \leq e^T(\frac{1}{2} (u^k_1, u^k_1)_{L^2} + \frac{1}{2}(u^k_0, u^k_0)_{H_*^2} + (V(u^k_0)-c(b-a)) \\ & + \frac{1}{2} \int_0^T (f(s),f(s))_{L^2} \, ds) 
\end{split}
\end{equation}

Due to equations \eqref{eq 1} and \eqref{eq 2} we have that $(u^k_1,u^k_1)_{L^2} \leq (u_1,u_1)_{L^2}$ and 
\begin{equation}\label{embedding}
(u^k_0,u^k_0)_{H_*^2} \leq (u_0,u_0)_{H_*^2}.
\end{equation}
for all $k \in \mathbb{N}$. Now recall that there exists a continuous embedding of $H^2_*(\Omega)$ into $C(\overline{\Omega})$.  We can then use \eqref{embedding} to conclude that $||u_0^k||_{C(\overline{\Omega})}$ is bounded with respect to $k$. Invoking Lemma 2.5, we have that $||f_2(u_0^k)||_{C(\overline{\Omega})}$ is bounded with respect to $k$, and hence $|V(u_0^k)|$ is bounded with respect to $k$. It follows that we can use \eqref{eq 9} to conclude that there exists a positive real number $C$ such that
\begin{equation*}
(u_k'(t),u_k'(t))_{L^2} + (u_k(t),u_k(t))_{H_*^2} \leq C.
\end{equation*} 
for any $t \in [0,t_k)$ and  $k \geq 1$. This uniform bound allows us to conclude that the solution $u_k(t)$ is globally defined on $[0,T]$ and the sequence $\{u_k\}_{k=1}^\infty$ is bounded in $C([0,T]; H_*^2(\Omega)) \cap C^1([0,T]; L^2(\Omega))$. 

{\it Step three.}
The goal of this step is to improve upon the uniform bounds on $\{u_k\}_{k=1}^\infty$ obtained in step two, Since $u_k \in C^3([0,T];  W_k)$;
\begin{equation*}
f \in C^1([0,T]; L^2(\Omega));
\end{equation*}
and $F_i$ is a $C^1$ map from $\mathbb{R}$ into $\mathbb{R}$ for all $i \in \{1,2\}$, we can differentiate \eqref{eq 6.5} with respect to time. Then, since $u_k''(t) \in W_k$ for all $t \in [0,T]$ we have
\begin{equation}\label{eq 10}
\begin{split}
& (u_k'''(t),u_k''(t))_{L^2} + (u_k'(t),u_k''(t))_{H^2_*} + (\frac{d F_1}{d u_k'} u_k''(t), u_k''(t))_{L^2} \\ &+  (\frac{d F_2}{d u_k} u_k'(t), u_k''(t))_{L^2} = (f'(t), u_k''(t))_{L^2}
\end{split}
\end{equation}
for any $t \in [0,T]$. 
Since we are assuming that $F_1$ is a nondecreasing function, we have
\begin{equation*}
 \frac{d F_1}{d u_k'} \geq 0,
\end{equation*}
for all values of its argument. Thus, we have
\begin{equation}\label{eq 10c}
\begin{split}
& (u_k'''(t),u_k''(t))_{L^2} + (u_k'(t),u_k''(t))_{H^2_*} +  (\frac{d F_2}{d u_k} u_k'(t), u_k''(t))_{L^2} \\ & \leq (f'(t), u_k''(t))_{L^2}
\end{split}
\end{equation}
This, in turn, allows us to write
\begin{equation}\label{eq 11}
\begin{split}
& \frac{1}{2} \frac{d}{dt}(u_k''(t),u_k''(t))_{L^2} + \frac{1}{2}\frac{d}{dt}(u_k'(t),u_k'(t))_{H^2_*} \\
& \leq |(\frac{d F_2}{d u_k} u_k'(t), u_k''(t))_{L^2}| + \frac{1}{2}(f'(t),f'(t))_{L^2} + \frac{1}{2}(u_k''(t),u_k''(t))_{L^2} 
\end{split}
\end{equation}
for any $t \in [0,T]$.

In step two we showed that $||u_k||_{C([0,T]; H_*^2(\Omega))}$ is bounded with respect to $k$. It follows that $||u_k||_{C([0,T];C(\overline{\Omega}))}$ is bounded with respect to $k$. Applying Lemma 2.5 and Remark 2.6 to $\frac{dF_2}{du_k}$, we see that $|| \frac{d F_2}{d u_k}||_{C([0,T]; C(\overline{\Omega}))}$ is bounded with respect to $k$. It follows that 
\begin{equation}\label{eq 12}
\begin{split}
& \frac{1}{2} \frac{d}{dt}(u_k''(t),u_k''(t))_{L^2} + \frac{1}{2}\frac{d}{dt}(u_k'(t),u_k'(t))_{H^2_*} \\
& \leq C(\frac{1}{2} ||u_k'(t)||^2_{L^2} + \frac{1}{2} ||u_k''(t)||^2_{L^2}) + \frac{1}{2}||f'(t)||^2_{L^2} + \frac{1}{2} ||u_k''(t)||^2_{L^2}
\end{split}
\end{equation}
for any $t \in [0,T]$ and  $k \geq 1$ for some constant $C$ independent of $k \geq 1$ and $t \in [0,T]$. 

Now recalling Lemma 2.2, we know that there exists a positive real number $C$ such that
\begin{equation*}
(w,w)_{L^2} \leq C(w,w)_{H_*^2}
\end{equation*}
for all $w \in H^2_*(\Omega)$. It follows that
\begin{equation}\label{eq 13}
\begin{split}
& \frac{1}{2}\frac{d}{dt} ||u_k''(t)||^2_{L^2} + \frac{1}{2}\frac{d}{dt}||u_k'(t)||^2_{H_*^2} \\
& \leq C(\frac{1}{2} ||u_k''(t)||^2_{L^2} + \frac{1}{2}||u_k'(t)||^2_{H_*^2}) + \frac{1}{2}||f'(t)||^2_{L^2}.
\end{split}
\end{equation}
for any $t \in [0,T]$ and $k \geq 1$ for some constant $C$ independent of $k \geq 1$ and $t \in [0,T]$. Invoking Gronwall's inequality, we see that there exists a positive real number $C$ such that for all $t \in [0,T]$ and $k \geq 1$
\begin{equation}\label{eq 14}
\begin{split}
& \frac{1}{2} ||u_k''(t)||^2_{L^2} + \frac{1}{2} ||u_k'(t)||^2_{H_*^2} \\
& \leq e^{Ct}(\frac{1}{2}||u_k''(0)||^2_{L^2} + \frac{1}{2} ||u_k'(0)||_{H_*^2}^2 + \int_0^t || f'(s)||^2_{L^2} \, ds) \\
& \leq e^{CT} (\frac{1}{2}||u_k''(0)||^2_{L^2} + \frac{1}{2} ||u_1^k||_{H_*^2}^2 + \int_0^T || f'(s)||^2_{L^2} \, ds) \\
&  \leq e^{CT} (\frac{1}{2}||u_k''(0)||^2_{L^2} + \frac{1}{2} ||u_1||_{H_*^2}^2 + \int_0^T || f'(s)||^2_{L^2} \, ds).
\end{split}
\end{equation}
An immediate consequence of \eqref{eq 14} is that there exists two positive real numbers $C_1$ and $C_2$ such that for all $t \in [0,T]$ and  $k \geq 1$
\begin{equation}\label{eq 16}
\begin{split}
& \frac{1}{2} ||u_k''(t)||^2_{L^2} + \frac{1}{2} ||u_k'(t)||^2_{H_*^2} \\
& \leq e^{C_1 T}(C_2 + (\frac{1}{2}||u_k''(0)||^2_{L^2})).
\end{split}
\end{equation}
\eqref{eq 16} tells us that all we need to do to obtain uniform bounds on $||u_k''(t)||_{L^2}$ and $||u_k'(t)||_{H_*^2}$ is to show that $||u_k''(0)||_{L^2}$ is bounded with respect to $k$. Towards this end let us recall \eqref{eq 5} and write
\begin{equation}\label{eq 17}
\begin{split}
& ||u_k''(0)||_{L^2}^2 = ||-L(u_k(0)) - P_k F_1(u_1^k) - P_k F_2(u_0^k) + P_k f(0)||_{L^2}^2 \\
& \leq 2(||L(u_k(0))||_{L^2}^2 + ||P_k F_1(u^k_1)||_{L^2}^2 + ||P_k F_2(u^k_0)||^2_{L^2} +||P_k f(0)||_{L^2}^2),\\
& \leq 2(||L(u_k(0))||_{L^2}^2 + || F_1(u^k_1)||_{L^2}^2 + || F_2(u^k_0)||^2_{L^2} +||f(0)||_{L^2}^2),\\
&\leq 2(||{u_k}(0)||_{H_*^4}^2 + || F_1(u^k_1)||_{L^2}^2 + || F_2(u^k_0)||^2_{L^2} +||f(0)||_{L^2}^2),\\
&\leq 2(||u^k_0||_{H_*^4}^2 + || F_1(u^k_1)||_{L^2}^2 + || F_2(u^k_0)||^2_{L^2} +||f(0)||_{L^2}^2), \\
& \leq 2(||u_0||_{H_*^4}^2 + || F_1(u^k_1)||_{L^2}^2 + || F_2(u^k_0)||^2_{L^2} +||f(0)||_{L^2}^2), \\ 
&\leq C + 2|| F_1(u^k_1)||_{L^2}^2 + 2|| F_2(u^k_0)||^2_{L^2}, 
\end{split}
\end{equation}
for any $k \geq 1$, where $C$ is a positive real number that does not depend on $k \geq 1$. Here we used \eqref{eq 1} to establish the next to last inequality in the above string of inequalities. It follows that all we need to do is to show that the second and third terms in the last line in the above string of inequalities are bounded with respect to $k$. To show that this is the case, let us recall again \eqref{eq 1} and \eqref{eq 2}. They imply that there exists a positive real number $C$ such that for all $k \geq 1$, $||u^k_0||_{H_*^4} + ||u^k_1||_{H_*^2} \leq C$. Since there exists a continuous embedding of $H_*^4(\Omega)$ and $H_*^2(\Omega)$ into $C(\overline{\Omega})$, we know that $||u^k_0||_{C(\overline{\Omega})} + ||u^k_1||_{C(\overline{\Omega})}$ is bounded with respect to $k$. Then Lemma 2.5 gives us that there exists a positive real number $C$ such that for all $k \geq 1$ 
\begin{equation}\label{eq 18}
|| F^2_1(u_1^k)||_{C(\overline{\Omega})} + ||F^2_{2}(u_0^k)||_{C(\overline{\Omega})} \leq C,
\end{equation}
Combining \eqref{eq 17} and \eqref{eq 18}, we see that there exists a positive real number $C$ such that for all $k \geq 1$ 
\begin{equation}\label{eq 19}
||u_k''(0)||_{L^2} \leq C.
\end{equation}
Recalling inequality \eqref{eq 16}, we can conclude that 
\begin{equation}\label{eq 20}
||u_k''(t)||_{L^2} + ||u_k'(t)||_{H_*^2} \leq C
\end{equation}
for any $t \in [0,T]$ and $k \geq 1$, where $C$ is a positive real number that does not depend on $t \in [0,T]$ or $k \geq 1$.

{\it Step four.} The goal of this step is to obtain a bound on
$$
(u_k(t),u_k(t))_{H_*^4}
$$
that is independent of $t \in [0,T]$ and $k \geq 1$. Since $L(u_k(t))$ is in $W_k$ for all $t \in [0,T]$ and $k \geq 1$, we can invoke \eqref{eq 6.5} and write 
\begin{equation}\label{eq 21}
(u_k(t), L(u_k(t)))_{H_*^2} = (f(t) - F_1(u_k'(t)) - F_2(u_k(t)) - u_k''(t), L(u_k(t)))_{L^2}, 
\end{equation}
for all $t \in [0,T]$. Now recall \eqref{eigenfunctions}. It implies that all of the elements of the sequence $\{e_i\}_{i=1}^\infty$ have the property that their second and fourth derivatives with respect to $x$ vanish on the boundary of $\Omega$. Since $u_k$ is in $W_k$ for all $t \in [0,T]$ and for all $k \in \mathbb{N}$, it follows that its second and fourth derivatives with respect to $x$ vanish on the boundary, for all $t \in [0,T]$ and for all $k \in \mathbb{N}$. We can now integrate by parts and obtain the following equation:
\begin{equation}\label{eq 22}
(u_k(t),u_k(t))_{H_*^4} = (f(t) - F_1(u_k'(t)) - F_2(u_k(t)) - u_k''(t), L(u_k(t)))_{L^2},
\end{equation}
for all $t \in [0,T]$. Applying the H\"older inequality to the right-hand side of the above equation, we have 
\begin{equation}\label{eq 23}
||u_k(t)||_{H_*^4} \leq ||f(t)- F_1(u_k'(t)) - F_2(u_k(t)) - u_k''(t)||_{L^2},
\end{equation}
for all $t \in [0,T]$ and $k \geq 1$. Now we can conclude that
\begin{equation}\label{eq 24}
||u_k(t)||_{H^4_*}^2 \leq 2(||f(t)||^2_{L^2} + ||F_1(u_k'(t))||^2_{L^2} + ||F_2(u_k(t))||_{L^2}^2 + ||u_k''(t)||^2_{L^2}),
\end{equation}
for all $t \in [0,T]$ and $k \geq 1$. 

Now recall that in steps 2 and 3 we showed that there exists a positive real number $C$ such that for all $k \geq 1$
\begin{equation}\label{eq 25}
\begin{split}
& ||u_k||_{C([0,T]; H^2_*(\Omega))} + ||u'_k||_{C([0,T]; H^2_*(\Omega))} \\ & + ||u_k''||_{C([0,T]; L^2(\Omega))} \leq C
\end{split}
\end{equation}
Since there exists a continuous embedding of $H^2_*(\Omega)$ into $C(\overline{\Omega})$, we also have that there exists a positive real number $C$ such that for all $k \geq 1$
\begin{equation}\label{eq 26}
||u_k||_{C([0,T]; C(\overline{\Omega}))} + ||u_k'||_{C([0,T]; C(\overline{\Omega}))} + ||u_k''||_{C([0,T]; L^2(\Omega))} \leq C.
\end{equation}
We can now invoke Lemma 2.5 and Remark 2.6  and see that there exists a positive real number $C$ such that for all $k \geq 1$
\begin{equation}\label{eq 27}
\begin{split}
& ||F^2_1(u_k')||_{C([0,T]; C(\overline{\Omega}))} +  ||F^2_{2}((u_k)||_{C([0,T]; C(\overline{\Omega}))}\\ & + ||u_k''||_{C([0,T]; L^2(\Omega))} \leq C.  
\end{split}
\end{equation}
Combining \eqref{eq 27} with \eqref{eq 24}, we can conclude that there exists a positive real number $C$ such that for all $k \geq 1$
\begin{equation}\label{eq 28}
||u_k||_{C([0,T]; H^4_*(\Omega))} \leq C.
\end{equation}

{\it Step five.} In steps 2-4 we obtained the existence of a positive real number $C$ such that
\begin{equation}\label{eq 29}
||u_k||_{C([0,T]; H_*^4(\Omega))} + ||u_k'||_{C([0,T]; H^2_*(\Omega))} + ||u_k''||_{C([0,T]; L^2(\Omega))}.
 \leq C.
\end{equation}
for all $k \geq 1$. It follows that there exists a subsequence $\{u_{k_l}\}_{l=1}^\infty$ of $\{u_k\}_{k=1}^\infty$ and a function $u$ that is in $W^{2,\infty}(0,T; L^2(\Omega))$, $W^{1,\infty}(0,T; H^2_*(\Omega))$, and $L^\infty(0,T; H^4_*(\Omega))$
such that
\begin{equation}\label{eq 30}
\begin{split}
& (a) u_{k_l}\; \text{converges weakly to}\;  u \; \text{with respect to the} \; L^2(0,T ;H_*^4(\Omega)) \; \text{norm} \\
& \text{as} \; l \rightarrow \infty. \\
& (b) u_{k_l}'\; \text{converges weakly to}\; u' \; \text{with respect to the} \; L^2(0,T;H^2_*(\Omega)) \; \text{norm} \\
& \text{as} \;  l \rightarrow \infty. \\
& (c) u_{k_l}''\; \text{converges weakly to}\; u'' \; \text{with respect to the} \; L^2(0,T ;L^2(\Omega)) \; \text{norm} \\
& \text{as} \;  l \rightarrow \infty.
\end{split}
\end{equation}

A consequence of the above is that we can assume without loss of generality that the above subsequence has the property that
\begin{equation}\label{compactness}
\begin{split}
& (a) u_{k_l}\; \text{converges to}\; u \; \text{with respect to the} \; C([0,T]; L^2(\Omega)) \; \text{norm} \\
& \text{as} \;  l \rightarrow \infty. \\
& (b) u_{k_l}'\; \text{converges to}\; u' \; \text{with respect to the} \; C([0,T]; L^2(\Omega)) \; \text{norm} \\
& \text{as} \;  l \rightarrow \infty. 
\end{split}
\end{equation}

Let us now observe that \eqref{eq 29} also implies that there exists a constant $C$ such that for all $l \geq 1$
\begin{equation}\label{eq 31}
||u_{k_l}||_{C([0,T];C(\overline{\Omega}))} + ||u_{k_l}'||_{C([0,T];C(\overline{\Omega}))} \leq C.
\end{equation}

We can now combine \eqref{compactness} and  \eqref{eq 31} with Lemma 2.7 to obtain 
\begin{equation}\label{eq 32}
\begin{split}
&(a) F_1(u_{k_l}') \rightarrow F_1(u') \; \text{with respect to the} \; C([0,T];L^2(\Omega)) \; \text{norm} \\ 
& \text{as} \; l \rightarrow \infty. \\
&(b) F_2(u_{k_l}) \rightarrow F_2(u) \; \text{with respect to the} \; C([0,T];L^2(\Omega)) \; \text{norm} \\ 
& \text{as} \; l \rightarrow \infty.
\end{split}
\end{equation}

{\it Step six.} In this step we show that the function $u(t)$, which is defined in step five, is a high regularity weak solution of \eqref{IBVP}. Proceeding as in section 7.2 of \cite{E}, we fix a positive integer $N$ and choose a function $v \in C^1([0,T];H^2_*(\Omega))$ of the form 
\begin{equation}\label{eq 33}
v(t) = \Sigma_{i=1}^N d_i(t) e_i,
\end{equation}
where $\{d_i\}_{i=1}^N$ are smooth functions. We select $k \geq N$, multiply \eqref{eq 6.5} by $d_i(t)$, sum $i=1,\dots, N$, and then integrate with respect to $t$, to discover
\begin{equation}\label{eq 34}
\begin{split}
& \int_{0}^T ((u_k''(t),v(t))_{L^2} + (u_k(t),v(t))_{H^2_*} + (F_1(u_k'(t)),v(t))_{L^2} \\ & + (F_2(u_k(t)),v(t))_{L^2}) \, dt =  \int_0^T (f(t),v(t))_{L^2}  \, dt.
\end{split}
\end{equation}

Next we set $k =k_l$ and use \eqref{eq 30}, \eqref{compactness}, and  \eqref{eq 32} to see that the function $u$ defined in step five satisfies the following equation:
\begin{equation}\label{eq 35}
\begin{split}
& \int_{0}^T ((u''(t),v(t))_{L^2} + (u(t),v(t))_{H^2_*} + (F_1(u'(t)),v(t))_{L^2} + \\ &  (F_2(u(t)),v(t))_{L^2}) \, dt = \int_0^T (f(t),v(t))_{L^2} \, dt.
\end{split} 
\end{equation}
This equation then holds for all $v \in L^2(0,T;H_*^2(\Omega))$, since functions of the form \eqref{eq 33} are dense in this space. This in turn allows us to conclude that $u$ satisfies \eqref{strong}. Finally, recalling \eqref{eq 1} and \eqref{eq 2}, we see that $u^k_0 \rightarrow u_0$ with respect to the $H_*^4(\Omega)$ norm and $u^k_1 \rightarrow u_1$ with respect to the $H^2_*(\Omega)$ norm. It follows that $u(0)=u_0$ and $u'(0) = u_1$. We can now conclude that $u$ is a high regularity weak solution of \eqref{IBVP}.

{\it Step seven.} In this step, we will show that the high regularity weak solution of \eqref{IBVP} constructed in step five is the unique high regularity weak solution of \eqref{IBVP}. Fix $t \in [0,T]$. Following arguments given in step six of this paper, we see that a high regularity weak solution $u$ of \eqref{IBVP} has the property that
\begin{equation}\label{integral}
\begin{split}
& \int_{0}^t ((u''(s),w(s))_{L^2} + (u(s),w(s))_{H^2_*} + (F_1(u'(s)),w(s))_{L^2} + \\ &  (F_2(u(s)),w(s))_{L^2}) \, ds = \int_0^t (f(s),w(s))_{L^2} \, ds
\end{split}
\end{equation}
for every $w \in L^2(0,t; H^2_*(\Omega))$. Now, let $u$ and $v$ be high regularity weak solutions of \eqref{IBVP}. Then we have
\begin{equation}\label{pre}
\begin{split}
& \int_{0}^t  ((u''(s) - v''(s), w(s))_{L^2} + (u(s)-v(s), w(s))_{H^2_*} + \\ & (F_1(u'(s))-F_1(v'(s)), w(s)))_{L^2} +  (F_2(u(s)) - F_2(v(s)), w(s))_{L^2}) \, ds=0,
\end{split}
\end{equation}
for all $w \in L^2(0,t;H^2_*(\Omega))$. 

We will now show that $u \equiv v$. Set $w = u' - v'$. Then we have
\begin{equation}\label{test}
\begin{split}
& \int_{0}^t  ((u''(s) - v''(s), u'(s) - v'(s))_{L^2} + (u(s)-v(s), u'(s) - v'(s))_{H^2_*} + \\
& (F_1(u'(s))-F_1(v'(s)), u'(s) - v'(s))_{L^2} + \\ 
& (F_2(u(s)) - F_2(v(s)), u'(s)-v'(s))_{L^2}) \, ds=0,
\end{split}
\end{equation}
 Calculus then gives us 
\begin{equation}\label{eq 37}
\begin{split}
& \int_0^t  (\frac{1}{2} \frac{d}{ds} ||u'(s)-v'(s)||_{L^2}^2 +  \frac{1}{2} \frac{d}{ds} ||u(s)-v(s)||_{H^2_*}^2 + \\
& (F_1(u'(s))-F_1(v'(s)), u'(s) - v'(s))_{L^2} + \\
& (F_2(u(s)) - F_2(v(s)), u'(s) - v'(s))_{L^2}) \, ds =0.
\end{split}
\end{equation}
Now, since $F_1$ is a nondecreasing function, we have
\begin{equation}\label{eq 38}
\begin{split}
& \int_0^t (\frac{1}{2} \frac{d}{ds} ||u'(s)-v'(s)||_{L^2}^2 + \frac{1}{2} \frac{d}{ds} ||u(s)-v(s)||_{H^2_*}^2  + \\ & (F_2(u(s)) - F_2(v(s)), u'(s) - v'(s))_{L^2}) \, ds  \leq 0.
\end{split}
\end{equation}
This in turn allows us to conclude that
\begin{equation}\label{eq 39}
\begin{split}
\int_0^t & (\frac{1}{2} \frac{d}{ds}||u'(s)-v'(s)||_{L^2}^2 + \frac{1}{2} \frac{d}{ds} ||u(s)-v(s)||_{H^2_*}^2) \, ds \\ & \leq \frac{1}{2} \int_0^t ||F_2(u(s)) - F_2(v(s))||_{L^2}^2 ds  + \frac{1}{2} \int_0^t ||u'(s) - v'(s)||_{L^2}^2 \, ds.
\end{split}
\end{equation}
An immediate consequence of the above is that
\begin{equation}\label{inequality}
\begin{split}
& \frac{1}{2} ||u'(t)-v'(t)||_{L^2}^2 + \frac{1}{2}||u(t)-v(t)||_{H^2_*}^2 \\ & \leq \frac{1}{2} \int_0^t ||F_2(u(s)) - F_2(v(s))||_{L^2}^2 \, ds  + \frac{1}{2} \int_0^t ||u'(s) - v'(s)||_{L^2}^2 \, ds.
\end{split}
\end{equation}

Now, recall that $F_2$ is a $C^1$ function from $\mathbb{R}$ into $\mathbb{R}$. It follows that there exists a positive real number $L_{F_2}$ 
such that
\begin{equation}\label{eq 40}
|F_2(s_1) - F_2(s_{2})| \leq L_{F_2}|s_1-s_2|,
\end{equation}
for all $s_1, s_{2} \in [-M,M]$, where $M$ is the maximum of the elements of the set
$$
\{||u(t)||_{C([0,T];C(\overline{\Omega}))},||v(t)||_{C([0,T];C(\overline{\Omega}))}\}.
$$
The definition of $M$, on the other hand, allows us to conclude that there exists a positive real number $C$ such that
\begin{equation}\label{eq 41}
||F_2(u(t)) - F_2(v(t))||^2_{L^2} \leq C||u(t)-v(t)||_{L^2}^2.
\end{equation}
for all $t \in [0,T]$. Note that we can assume without loss of generality that $C \geq 1$, so let us assume that this is the case. We now have
\begin{equation}\label{eq 42}
\begin{split}
& \frac{1}{2} ||u'(t)-v'(t)||_{L^2}^2 + \frac{1}{2} ||u(t)-v(t)||_{H^2_*}^2  \\ & \leq \int_0^t \frac{1}{2}C||u(s) -v(s)||_{L^2}^2  \, ds + \int_0^t  \frac{1}{2}C|| u'(s) - v'(s))||_{L^2}^2 \,  ds
\end{split}
\end{equation}
for all $t \in [0,T]$. We can now invoke Lemma 2.2 to conclude that there exists a positive constant $C$ such that
\begin{equation}\label{eq 43}
||u(t)-v(t)||^2_{H^2_*} \geq C ||u(t)-v(t)||^2_{L^2},
\end{equation}
for all $t \in [0,T]$. Again, we can assume that $C$ is greater than one, so let us do so. It follows that there exists a positive constant $C$ such that
\begin{equation}\label{eq 44}
\begin{split}
& \frac{1}{2} ||u'(t)-v'(t)||_{L^2}^2 + \frac{1}{2} ||u(t)-v(t)||_{H^2_*}^2  \\ & \leq  \int_0^t \frac{1}{2}C|| u'(s) - v'(s))||_{L^2}^2 \, ds + \int_0^t \frac{1}{2}C||u(s) -v(s)||_{H^2_*}^2 \, ds 
\end{split}
\end{equation}
for all $t \in [0,T]$. We can now invoke Gronwall's inequality to conclude that $u(t) \equiv v(t)$ for all $t \in [0,T]$. Uniqueness of high regularity weak solutions of \eqref{IBVP} follows.

\end{proof}

\bibliographystyle{amsplain}

\end{document}